\documentclass[12pt,twoside]{article}
\usepackage{amssymb,latexsym}
\usepackage{enumerate}
\usepackage{verbatim}
\usepackage{amsfonts}
\usepackage{amsmath}
\usepackage{fancyhdr}
\usepackage{amsmath,amsxtra,amssymb,amsthm,latexsym,amscd}
\usepackage{array}
\newtheorem{thm}{Theorem}[section]

\newtheorem{lem}[thm]{Lemma}

\newcommand{\ds}{\displaystyle}

\theoremstyle{definition}
\newtheorem{df}[thm]{Definition}
\newtheorem{rem}[thm]{Remark}

\numberwithin{equation}{section}

\newcommand{\R}{\mathbb{R}}
\newcommand{\N}{\mathbf{N}}
\newcommand{\rank}{{\textrm{rank}}}

\date{}

\begin{document}
\centerline{\bf Int. Journal of Math. Analysis, Vol. 6, 2012, no. 10, 481 - 491}

\centerline{}

\centerline{}

\centerline {\Large{\bf The Quantitative Morse Theorem}}

\centerline{}

\centerline{\bf {Ta L\^e Loi}}

\centerline{}

\centerline{University of Dalat, Dalat, Vietnam}

\centerline{loitl@dlu.edu.vn}

\centerline{}

\centerline{\bf {Phan Phien}}

\centerline{}

\centerline{Nhatrang College of Education, Nhatrang, Vietnam}

\centerline{phieens@yahoo.com}





\newtheorem{Theorem}{\quad Theorem}[section]

\newtheorem{Definition}[Theorem]{\quad Definition}

\newtheorem{Corollary}[Theorem]{\quad Corollary}

\newtheorem{Lemma}[Theorem]{\quad Lemma}

\newtheorem{Example}[Theorem]{\quad Example}

\begin{abstract}
In this paper, we give a proof of the quantitative Morse theorem stated by \mbox{Y. Yomdin} in \cite{Y1}. The proof is based on the quantitative Sard theorem, the quantitative inverse function theorem and the quantitative Morse lemma.
\end{abstract}

{\bf Mathematics Subject Classification:}  Primary 58K05; Secondary 58E05, 97N40\\

{\bf Keywords:} Morse theorem, Quantitative assessment, Critical and near-Critical point

\section{Introduction}
\renewcommand{\thefootnote}{}
\footnote{This research is supported  by Vietnam's National Foundation for Science and Technology Development (NAFOSTED).}
One of the first basic results of classical singularity theory are that Sard Theorem \cite{Sa} and \cite{Morse1}, and Morse theorem \cite{Morse}. These theorems research critical points and critical values of smooth mappings on open subsets of $\R^n$.
The quantitative assessments and applications of the theorems were also considered. \mbox{Y. Yomdin} in \cite{Y2} introduced the concept of near-critical points and near-critical values of a map, and there have been many results on quantitative assessments for the set of these points and values.
One of them is the quantitative Sard theorem for mappings of class $C^k$ (see \cite{Y1}, \cite{Y2}, \cite{Y3}, \cite{Y-C} and \cite{R}). The results give some explicit bounds in term of $\varepsilon$-entropy of the set of near-critical values.
For Morse theorem, in \cite{Y1} \mbox{Y. Yomdin}  also stated a quantitative form for $C^k$-functions. But in the article, he just gave a few suggestions without details for the proof of the theorem. Up to now,  he  probably hasn't published the proof.

 In this paper, we give a detailed proof of  the quantitative Morse theorem. The proof is based on the quantitative Sard theorem, the quantitative inverse function theorem and the quantitative Morse lemma.
\section{Preliminaries}
We give here some definitions, notations and results that will be used later.

\vspace*{8pt}\noindent
Let $\mathbf{M}_{m\times n}$ denote the vector space of real $m \times n$ matrices,
\begin{itemize}
\item[] $\|x\| =(|x_1|^2 + \cdots + |x_n|^2)^{\frac{1}{2}}$, ~where $x=(x_1, \ldots, x_n) \in \mathbb{R}^n$,\\
$\mathbf{B}^n$ denotes the unit ball in $ \mathbb{R}^n$, $\mathbf{B}_r^n$ denotes the ball of radius $r$, centered at $0\in \R^n$, and $\mathbf{B}_r^n(x_0)$ denotes the ball of radius $r$, centered at $x_0\in \R^n$,
\item[] $\|A\| = \max_{\|x\|=1}\|Ax\|, ~$where $A \in \mathbf{M}_{m\times n}$,
\item[] $\ds\|A\|_{\max} = \max_{i, j}|a_{ij}|$, where $A=(a_{ij})_{m\times n}\in  \mathbf{M}_{m\times n}$,
\item[] $\mathcal{B}_{n\times n}$ denotes the unit ball in $\mathbf{M}_{n\times n}$,
\item[] $\textrm{Sym}(n)$ denotes the space of real symmetric  $n\times n$-matrices.
\end{itemize}
\begin{df}
Let $f: M \rightarrow \mathbb{R}^m$ be a differentiable mapping class $C^k$, $M \subset \mathbb{R}^n$. Then $C^k$-norm of $f$ is defined by
\[\|f\|_{C^k} =  \sum_{j=1}^k\sup_{x \in M}\|D^jf(x)\|.\]
\end{df}
\begin{df}\label{dn1}
A mapping $ f: \mathbb{R}^m  \rightarrow  \mathbb{R}^n $ is called \textbf{Lipschitz} in a neighborhood of a point $ x_0 $ in $ \mathbb{R}^n $  if there exists a constant $ K>0 $ such that for all $x$ and $y$ near $ x_0 $, we have
\[
\|f(x) - f(y)\|\leq K\|x-y\|.
\]
Then we call $f$ the \textbf{$K$-Lipschitz}.
 \end{df}
\noindent The usual $m \times n$ Jacobian matrix of partial derivatives of $f$ at $x$, when it exists, is denoted by $Jf(x)$. By Rademacher's theorem (see \cite[Theorem 3.1.6]{F}), we have the following definition:
\begin{df}[F. H. Clarke - \textrm{[C1], [C2]}]\label{dn2}  The \textbf{generalized Jacobian} of $f$ at $x_0$, denoted by $\partial f(x_0)$, is the convex hull of all matrices $M$ of the form
\[M = \lim_{i\rightarrow \infty}Jf(x_i),\]
where $f$ is differentiable at $x_i$ and $x_i$ converges to $x_0$ for each $i$.\\
$\partial f(x_0)$ is said to be of \textbf{maximal rank} if every $M$ in $\partial f(x_0)$ is of maximal rank.
\end{df}
\begin{thm}[Quantitative inverse function theorem, c.f. \cite{C1} and \cite{P}]\label{dl2}\textit{ }\\
Let $f: \mathbb{R}^n \rightarrow\mathbb{R}^n$ be a $K$-Lipschitz mapping in a  neighborhood of a point $ x_0 $ in $ \mathbb{R}^n $. Suppose that $~\partial f(x_0)$ is of maximal rank, set
\[\delta = \frac{1}{2}\inf_{M\in \partial f(x_0)}\frac{1}{\|M^{-1}\|},\]
$r$ be chosen so that $f$ satisfies $K$-Lipschitz condition  and
$\partial f(x) \subset \partial f(x_0) + \delta \mathcal{B}_{n\times n},~~ \textrm{when}~~x \in \mathbf{B}_r^n(x_0).$ Then $f$ is inversible in $\mathbf{B}_{\frac{r\delta}{2K}}^n(x_0)$ and
there  exists the inverse mapping
\[g: \mathbf{B}_{\frac{r\delta}{2}}^n(f(x_0)) \rightarrow \R^n\]
being $\ds\frac{1}{\delta}$-Lipschitz.
\end{thm}
\begin{df}[Singular values of linear mapping, c.f. \cite{G-L}] Let $L: \mathbb{R}^n \rightarrow \mathbb{R}^m$ be a linear mapping. Then there exist $\sigma_1(L) \geq \ldots \geq \sigma_r(L)>0$, where $r = \rank L$, so that $L(\mathbf{B}^n)$ is an $r$-dimensional ellipsoid of semi-axes $\sigma_1(L) \geq \ldots \geq \sigma_r(L)$. Set $\sigma_0(L) = 1$ and  $\sigma_{r+1}(L) = \ldots = \sigma_m(L) = 0$, when $r < m$.\\
We call $\sigma_0(L), \ldots, \sigma_m(L)$ the\textbf{ singular values} of $L$.
\end{df}
\begin{rem}\label{nx221} Let $L$ be a linear mapping or a matrix. Then
\begin{itemize}
\item[($i$)] $\sigma_{\max}(L)= \|L\| = \sigma_1(L)$, \ \
$\sigma_{\min}(L)= \min_{\|x\|=1}\|Lx\|$.
\item[($ii$)]  When $L \in L(\mathbb{R}^n, \mathbb{R}^n)$, and $\lambda$ is a  eigenvalue of $L$, we have
\[\ds\sigma_{\min}(L) \leq |\lambda| \leq \sigma_{\max}(L).\]
\end{itemize}
\end{rem}
\begin{df}\label{dn2213}Let $f: \mathbb{R}^n \rightarrow \mathbb{R}^m$ be a $k$ times differentiable mapping, $k \geq 1$. For $\Lambda = (\lambda_1, \ldots, \lambda_m)$,  $\lambda_1 \geq \ldots \geq \lambda_m \geq 0$, we call
\[\Sigma(f, \Lambda) = \{x\in \R^n: \sigma_i(Df(x)) \leq \lambda_i, i = 1, \ldots, m\}\]
the set of \textbf{$\Lambda$-critical points} of $f$, and
\[\Delta(f, \Lambda) = f(\Sigma(f, \Lambda))\]
the set of \textbf{$\Lambda$-critical values} of $f$.\\
Set $\Sigma(f, \Lambda, A) = \Sigma(f, \Lambda) \cap A, \ \Delta(f, \Lambda, A) = f(\Sigma(f, \Lambda, A)),  \ A \subset \R^n$. When $\gamma = (\gamma, \ldots, \gamma)\in \R^m_+$, a point $y \in \R^m$ is called \textbf{$\gamma$-regular value} of $f$ if $y \notin \Delta(f, \gamma, A)$, i.e $f^{-1}(y)= \emptyset$ or if $x\in  f^{-1}(y)$ then there exists a number $i \in \{1, \ldots, m\}$ so that $\sigma_i(Df(x)) \geq \gamma$.
\end{df}
\begin{rem}
If $\Lambda = (0, \ldots, 0)$ then $\Sigma(f, 0)$ is the set of critical points and $\Delta(f, 0)$ is the set of critical values of $f$.
\end{rem}
\begin{df}\label{dn2221} Let $X$ be a metric space, $A \subset X$ a relatively compact subset. For any $\varepsilon > 0$, denoted by $M(\varepsilon, A)$ the minimal number of closed balls of radius $\varepsilon$ in $X$, covering $A$. 
\end{df}
\begin{thm}[Quantitative Sard theorem, c.f. {\cite[Theorem 9.6]{Y-C}}]\label{dl222}
Let $f: \mathbf{B}_r^n \rightarrow \mathbb{R}^m$ be a mapping of class $C^k$, $q=\min(n, m)$, $\Lambda = (\lambda_1, \ldots, \lambda_q)$, $\lambda_i>0, i = 1, \ldots, q$, $\mathbf{B}_\delta^m$ is a ball of radius $\delta$ in $\mathbb{R}^m$. When $0<\varepsilon \leq \delta$
\[M(\varepsilon, \Delta (f, \Lambda, \mathbf{B}_r^n) \cap \mathbf{B}_\delta^m) \leq c\left(  \frac{R_k(f)}{\varepsilon}\right)^{\frac{n}{k}} \sum_{i=0}^q \min \left( \lambda_0 \ldots \lambda_i\left(  \frac{r}{\varepsilon}\right)^i \left( \frac{\varepsilon}{R_k(f)}\right) ^{\frac{i}{k}}, \left(\frac{\delta}{\varepsilon} \right)^i\right), \]
where $c=c(n, m, k)$, $R_k(f) = \frac{K}{(k-1)!}r^{k-1}$, and $K$ is a Lipschitz constant of $D^{k-1}f$ in $\mathbf{B}_r^n$.
\end{thm}
\begin{lem}[Quantitative Morse lemma]\label{md241}
Let $A \in Sym(n)$. Suppose that $Q_0 \in Gl(n)$ such that $^tQ_0AQ_0 = D_0 = diag(1, \ldots, 1, -1, \ldots, -1)$. Set
\[U(A) = \{B \in Sym(n): \|B - A\| \leq \frac{1}{2n\|Q_0\|^2}\}.\]
Then there exists a mapping $\mathcal{P}: U(A) \rightarrow Gl(n) \in C^\omega$ satisfying
\[\mathcal{P}(A) =Q_0, \textrm{and if}\ \mathcal{P}(B) = Q \ \textrm{then} \ ^tQBQ = D_0.\]
\end{lem}
\begin{proof} For $B \in U(A)$, we have
\[\begin{array}{rcl}
\|^tQ_0BQ_0 - ^tQ_0AQ_0\|_{\max} & \leq & \|^tQ_0(B - A)Q_0\|\\
&\leq&\|Q_0\|^2\|B - A\|\\
&\leq& \ds\frac{1}{2n}.
\end{array}\]
If  $^tQ_0BQ_0 = (b_{ij})_{1\leq i, j \leq n}$ then $|b_{ii}| > \ds\sum_{j \neq i}|b_{ij}|$.
So  $\ds\det(b_{ij})_{1\leq i, j \leq k} \neq 0$, for $k = 1, \ldots, n$. Therefore,  the normalization (see \cite[Lemma p.145]{Hir}) reductioning  $^tQ_0BQ_0$ to the normal form $D_0$ defines the mapping $\mathcal{P}: U(A) \rightarrow \textrm{Gl}(n) \in C^\omega$ satisfying the demands of the lemma.
\end{proof}
\begin{rem}\label{nx224} The reduction a non-degenerate real symmetric matrix $A$ to the normal form $D_0$ can be realized by a matrix $Q_0$ of the form $Q_0=SU$, where $U$ is a orthogonal matrix, and $S$ is a diagonal matrix. So
\[\|Q_0\|^2 =\frac{1}{\sigma_{\min}(A)}.\]
\end{rem}
\section{The quantitative Morse theorem}
\begin{thm}[c.f. {\cite[Theorem 4.1, Theorem 6.1]{Y1}}]\label{dl2411}
Fix $k\geq$ 3. Let $f_0: \overline{\mathbf{B}}^n \rightarrow \mathbb{R}$ be a $C^k$-function in a open set contain $\overline{\mathbf{B}}^n$ with all derivatives up to order $k$ uniformly bounded by $K$. Then for any given $\varepsilon > 0$, we can find $h$ with $\|h\|_{C^k} \leq \varepsilon$ and the positive functions $\psi_1$, $\psi_2$, $\psi_3$, $d$, $M$, $N$, $\eta$ depending on $K$ and $\varepsilon$, such that $f = f_0 + h$ satisfies the following conditions:
\begin{itemize}
\item [$(i)$] At each critical point $x_i$ of $f$, the smallest absolute value of the eigenvalues of the Hessian $Hf(x_i)$ is at least $\psi_1(K, \varepsilon)$.
\item [$(ii)$] For any two different critical points $x_i$ and $x_j$ of $f$, $\|x_i-x_j\| \geq d(K, \varepsilon)$.  Consequently, the number of the critical points does not exceed $N(K, \varepsilon)$.
\item [$(iii)$]  For any two different critical points $x_i$ and $x_j$ of $f$, $|f(x_i)-f(x_j)| \geq \psi_2(K, \varepsilon)$.
\item [$(iv)$] For $\delta = \psi_3(K, \varepsilon)$ and for each critical point $x_i$ of $f$, there exists a coordinate transformation  $\varphi: \mathbf{B}_\delta^n(x_i) \rightarrow \R^n \in C^r$ such that
\begin{displaymath}
f\circ\varphi^{-1}(y_1,\ldots, y_n) = y_1^2 + \cdots + y_l^2 - y_{l+1}^2 - \cdots -
y_n^2 + const,
\end{displaymath}
and $\|\varphi\|_{C^{k-1}} \leq M(K, \varepsilon)$.
\item [$(v)$]  If  $\|\textrm{grad}f(x)\| \leq \eta(K, \varepsilon)$, then $x\in \mathbf{B}_\delta^n(x_i)$, with $x_i$ is a critical point of $f$.
\end{itemize}
\end{thm}
The proof of $(i)$ is based on the suggestion of Y. Yomdin (see \cite{Y-C}). The proofs of $(ii), (iii), (iv)$ and $(v)$ are based on the quantitative inverse theorem and the quantitative  Morse lemma in section 2.

\begin{proof}\textit{ }\\
($i$)
 Let $ \varepsilon >0$. Applying Theorem \ref{dl222},
\[M(r, \Delta(Df_0, \gamma, \overline{\mathbf{B}}^n) \cap \mathbf{B}_\varepsilon^n)\leq cR_k(f_0)^\frac{n}{k}\frac{1}{{r}^\frac{n}{k}}\ds\sum_{i = 0}^n
\left(\ds\frac{\gamma}{ R_k(f_0)^\frac{1}{k}r^\frac{k-1}{k}}\right)^i.\]
When $r< 1$ and $\gamma<r R_k(f_0)^\frac{1}{k}$,
\[
\begin{array}{rcl}
M(r, \Delta(Df_0, \gamma, \overline{\mathbf{B}}^n) \cap \mathbf{B}_{\varepsilon}^n)&\leq& cR_k(f_0)^\frac{n}{k}\frac{1}{{r}^\frac{n}{k}}\ds\sum_{i = 0}^n\left(\ds\frac{\gamma}{r R_k(f_0)^\frac{1}{k}}\right)^i\\
&\leq& cR_k(f_0)^\frac{n}{k}\frac{1}{{r}^\frac{n}{k}}\frac{1}{1-\frac{\gamma}{r R_k(f_0)^\frac{1}{k}}}.
\end{array}
\]
So the Lebesgue measure of $\Delta(Df_0, \gamma, \overline{\mathbf{B}}^n) \cap \mathbf{B}_\varepsilon^n$,
\[\begin{array}{rcl}
 m(\Delta(Df_0, \gamma, \overline{\mathbf{B}}^n) \cap \mathbf{B}_\varepsilon^n)
&\leq &r^nm(\overline{\mathbf{B}}^n)M(r,\Delta(Df_0, \gamma, \overline{\mathbf{B}}^n) \cap \mathbf{B}_\varepsilon^n)\\
&\leq& r^nm(\overline{\mathbf{B}}^n)cR_k(f_0)^\frac{n}{k}\frac{1}{{r}^\frac{n}{k}}\frac{1}{1-\frac{\gamma}{r R_k(f_0)^\frac{1}{k}}}.
\end{array}
\]
Let 
$$r(\varepsilon)=\frac 12\min( \varepsilon,( \frac{\varepsilon}{c^\frac{1}{n}R_k(f_0)^\frac{1}{k}})^\frac{k}{k-1}),$$
and
$$\gamma(K, 2\varepsilon)= R_k(f_0)^\frac{1}{k}r(\varepsilon)(1-\frac{r(\varepsilon)^ncR_k(f_0)^\frac{n}{k}}{\varepsilon^nr(\varepsilon)^\frac{n}{k}})>0,$$
with $R_k(f_0) = \frac{K}{(k-1)!}$. We get
 \[ m(\Delta(Df_0, \gamma, \overline{\mathbf{B}}^n) \cap \mathbf{B}_\varepsilon^n)
< \varepsilon^nm( \overline{\mathbf{B}}^n)=m( \mathbf{B}_\varepsilon^n).\]
 So we can choose a $\gamma(K, 2\varepsilon) $-regular value $v$ of $Df_0$, with $\|v\|<\varepsilon$.

Now, let $h: \overline{\mathbf{B}}^n \rightarrow \R$ be a linear mapping with $Dh = -v$ and $f=f_0+h$. Then
$\|h\|_{C^k} \leq \varepsilon$, $Df=Df_0-v$, and $Hf=Hf_0=D(Df_0)$. So
$0$ is a $\gamma(K, 2\varepsilon)$-regular value of $Df$, and  at each critical point $x_i $ of $f$, we have

\begin{equation}\label{ct241}
\|Hf(x_i)\| \geq \gamma(K, 2\varepsilon).
\end{equation}
By Remark \ref{nx221}, the smallest absolute value of the eigenvalues of the Hessian $Hf(x_i)$ is at least $\psi_1(K, 2\varepsilon) = \gamma(K, 2\varepsilon).$

\vspace{8pt}\noindent
($ii$) Consider $Df: \overline{\mathbf{B}}^n \rightarrow \R^n$. Suppose that $x_i$ is a critical point of  $f$. Then applying (\ref{ct241}) we obtain
\[\delta' = \frac{1}{2}\frac{1}{\|Hf(x_i)^{-1}\|}\geq \frac{1}{2}\gamma(K, 2\varepsilon).\]
Choose
$\delta' = \frac{1}{2}\gamma(K, 2\varepsilon)$,
we have
\[\|D(Df)(x) - D(Df)(x_i)\| = \|D(Df_0)(x) - D(Df_0)(x_i)\| \leq K \|x -x_i\|,\]
hence, if $\ds \|x -x_i\| \leq \frac{\delta'}{K}$, we get
$$D(Df)(x) \in D(Df)(x_i) + \delta' \mathcal{B}_{n\times n}.$$
Therefore, if $\ds r = \frac{\delta'}{K}$, then every $x \in \mathbf{B}_r^n(x_i)$ we obtain
\[D(Df)(x) \in D(Df)(x_i) + \delta' \mathcal{B}_{n\times n}.\]
Thus, applying Theorem \ref{dl2}, $Df$ is invertible in
$$\mathbf{B}_{\frac{r\delta'}{2K}}^n\left( x_i\right)  = \mathbf{B}_{\frac{\gamma^2(K, \varepsilon)}{8K^2}}^n\left( x_i \right).$$
Hence, $Df^{-1}(0)$ is unique in $\ds \mathbf{B}_{\frac{\gamma^2(K, \varepsilon)}{8K^2}}^n\left( x_i \right)$, i.e. $x_i$ is the unique critical point of $f$ in the ball  $\ds\mathbf{B}_{\frac{\gamma^2(K, \varepsilon)}{8K^2}}^n\left( x_i \right)$.\\
So if $x_i, x_j$ are different critical points of $f$, we have
\[d(x_i, x_j) \geq d(K, 2\varepsilon) = \frac{1}{4}\frac{\gamma^2(K, 2\varepsilon)}{K^2} > 0.\]
Therefore, the number of critical points $x_i$ does not exceed \\
$\textrm{ }\ \ \ \ \ \ \ \ \ \ \ \ \ \ \ \ \ \ \ \ \ \ \  \ \ \ \ \ \ \ N(K, 2\varepsilon) =$ $ M\left(\frac{1}{4}\frac{\gamma^2(K, 2\varepsilon)}{K^2}, \overline{\mathbf{B}}^n\right)$.

\vspace{8pt}\noindent
($iii$) Suppose that the number of critical points of $f$ being $N, N \leq N(K, 2\varepsilon)$, and critical values of $f$ ordered as follows:
$$f(x_1) \leq f(x_2) \leq \ldots \leq f(x_{N}).$$
For each critical point $x_i$ of $f$, set
\[U_i = \mathbf{B}_{\frac{d(K, 2\varepsilon)}{2}}^n(x_i) \cap \overline{\mathbf{B}}^n, \ \ \ B_i=\mathbf{B}_{\frac{d(K, 2\varepsilon)}{4}}^n(x_i) \cap \overline{\mathbf{B}}^n.\]
We call $\lambda_i: \overline{\mathbf{B}}^n \rightarrow [0, 1]$ the mapping of class $C^\infty$,
where
\[\lambda_i(x) = \left\lbrace \begin{array}{rl}
0, &x \notin U_i\\
1,&x \in B_i
\end{array}
\right.\]
with all derivatives uniformly bounded by $C_1$. Set ~~$\widetilde{f} = f + \lambda$, with
\[\lambda: \overline{\mathbf{B}}^n \rightarrow \R,~~~ \ds \lambda(x) = \sum_{i=1}^{N}c_i\lambda_i(x), \ \textrm{where} \ c_i = i \cdot \frac{\varepsilon}{2C_1kN^2} > 0.\]
From $(ii)$ we obtain every $U_i$ disjoint, and we have
$\|\lambda\|_{C^k} \leq \frac{\varepsilon}{2}.$\\
Thus $\widetilde{f}$ will be a Morse function having the same critical points as $f$ and these will have the same indices. Moreover,
$\widetilde{f}(x_i) = f(x_i) + c_i$. Hence, with $x_i, x_j$ are critical points, $i \neq j$, we obtain
\begin{equation}
\label{ct242}
|\widetilde{f}(x_i) - \widetilde{f}(x_j)| = |f(x_i) + c_i - f(x_j) - c_j| \geq \frac{\varepsilon}{2kC_1N^2} > 0.
\end{equation}
Therefore, replacing the linear mapping $h$ in  ($i$)  by
$$h = h_1 + \lambda,$$ with $h_1: \overline{\mathbf{B}}^n \rightarrow \R$ being a linear mapping such that $Dh_1 = -v$, and $v$ is a $\gamma(K, \varepsilon)$-regular value of $Df_0$, at a distance of most $\frac{\varepsilon}{2}$  from $0$, we get
$$\ds \|h\|_{C^k} = \|h_1 + \lambda\|_{C^k} \leq \varepsilon,$$
and $f = f_0 + h = f_0 + h_1 + \lambda$ to satisfy $(i)$ and $(ii)$, with
\[\ds\psi_1(K, \varepsilon) = \gamma(K, \varepsilon); \
\ds d(K, \varepsilon) = \frac{1}{4}\frac{\gamma^2(K, \varepsilon)}{K^2}, \ N(K, \varepsilon) = M\left(\frac{1}{4}\frac{\gamma^2(K, \varepsilon)}{K^2}, \overline{\mathbf{B}}^n\right).\]
Moreover, by (\ref{ct242}), for any $i \neq j$, we have
\[|f(x_i)- f(x_j)|\geq\psi_2(K, \varepsilon) = \frac{\varepsilon}{2kC_1N^2(K, \varepsilon)} > 0.\]

\vspace*{8pt} \noindent
($iv$) According to ($ii$),  we only need to prove ($iv$) for each critical point $x_i$. Moreover, we may assume $x_i=0, f(x_i)=0$.\\
Let $Q_0 \in \textrm{Gl}(n)$ be a linear transformation satisfying the condition of Remark \ref{nx224} such that
\[^tQ_0Hf(0)Q_0 = D_0.\]
The coordinate transformation $\varphi$ is constructed  as follows.\\
First, let
$\mathcal{B}: \overline{\mathbf{B}}^n \rightarrow \textrm{Sym}(n) \in C^{k-1}$ in a open set contain $\overline{\mathbf{B}}^n$, be defined by
\[\mathcal{B}(x) = B_x= (b_{ij}(x))_{1 \leq i, j \leq n},\]
where
\[b_{ij}(x) = \int_0^1 \int_0^1\frac{\partial^2 f}{\partial x_j \partial x_i}(stx)ds dt, 1\leq i, j \leq n.\]
Then
$$f(x) = \ds\sum_{i, j =1}^nb_{ij}(x)x_i x_j ~~~\textrm{and}~~\mathcal{B}(0) = A = Hf(0).$$
Applying Lemma \ref{md241},
we get
\[\mathcal{P}: U(A) \rightarrow \textrm{Gl}(n),\]
being of class $C^\omega$ such that $\mathcal{P}(A) = Q_0$, and if $\mathcal{P}(B) = Q$ then $^tQBQ = D_0$.\\
According to the Mean Value Theorem and Remark \ref{nx224}, the condition to apply Lemma \ref{md241} is
\[\|Hf(x)-A\| \leq (K + \varepsilon)\|x\| \leq \frac{1}{2n\|Q_0\|^2} = \frac{1}{2n}\sigma_{\min}(A),\]
or
\[\|x\| \leq \frac{1}{(K+\varepsilon)2n\|Q_0\|^2}  = \frac{1}{(K+\varepsilon)2n}\sigma_{\min}(A).\]
Set $\delta = \psi_3(K, \varepsilon) = \frac{1}{(K+\varepsilon)2n}\gamma(K,\varepsilon)$ and
\[\varphi: U_{\delta}(0) \rightarrow \R^n, \ \ \ y = \varphi(x) = Q_x^{-1}x, \ \ \ \textrm{with} \ Q_x = \mathcal{P}(B_x).\]
We have
\[f(x) ={^tx}B_xx={^ty}(^tQ_xB_xQ_x)y={^ty}D_0y = y_1^2 + \cdots + y_l^2 - y_{l+1}^2 - \cdots -
y_n^2.\]
To prove $\|\varphi\|_{C^{k-1}}\leq M(K, \varepsilon)$, present $\varphi$ as the following composition
\[\varphi: x \in U_\delta(0) \xrightarrow{\mathcal{B}} B_x \xrightarrow{\mathcal{P}} Q_x \xrightarrow{\textrm{Inv}} Q_x^{-1} \xrightarrow{L} \varphi (x) = Q_x^{-1}x.\]
By the construction $\mathcal{B}\in C^{k-1}$, and
by the assumption, the partial derivatives of $\mathcal{B}$
\[\|\partial^\alpha \mathcal{B}(x)\| \leq K, \text{for all } \alpha\in \N^n, |\alpha| \leq k -1.\]
Since $U(A)$ is compact, there exists $M_1(K, \varepsilon) >0$ such that
\[\|\partial^\alpha \mathcal{P}(B)\| \leq M_1(K, \varepsilon), \ \textrm{for all} \ B \in U(A), |\alpha| \leq k-1.\]
Similarly, since $\mathcal{P}(U(A))$ is compact, there exists $C_2(K, \varepsilon)>0$ such that
\[\|\partial^\alpha \textrm{Inv}(Q)\| \leq C_2(K, \varepsilon), |\alpha| \leq k-1, \ \textrm{for all}\ Q \in \mathcal{P}(U(A)).\]
Let
\[\overline{L}: U_\delta \times \textrm{Inv}(\mathcal{P}(U(A))) \rightarrow \R^n, \overline{L}(x, Q') = Q'x.\]
Then $\overline{L}$ is a bilinear form. Hence there exists $C_3(K, \varepsilon)>0$ such that
\[\|\partial L\| \leq C_3(K, \varepsilon),\ \|\partial^\alpha L\| = 0, \ \textrm{for} \ |\alpha| \geq 2, \ \textrm{and} \ (x, Q') \in U_\delta \times \textrm{Inv}(\mathcal{P}(U(A))).\]
Since $\partial^\alpha \varphi$ can be represented as a sum of products of $\partial^{\alpha_1}\mathcal{B}, \partial^{\alpha_2}\mathcal{P}, \partial^{\alpha_3}\textrm{Inv}$ and $\partial^{\alpha_4}L$, with $|\alpha_j|\leq |\alpha|, j = 1, \ldots, 4$, there exits $M(K, \varepsilon)>0$ depending on $K$, $M_1(K, \varepsilon)$, $C_2(K, \varepsilon)$ and $C_3(K, \varepsilon)$ such that $\|\varphi\|_{C^{k-1}} \leq M(K, \varepsilon)$.

\vspace*{8pt}\noindent
($v$)
Consider $Df: \overline{\mathbf{B}}^n \rightarrow \R^n$. Then for $x_i$ is critical point of $f$, we have
\[Df(x_i) = 0, \|Df(x)\| = \|Df(x) - Df(x_i)\| ~~\textrm{for all} ~x \in \overline{\mathbf{B}}^n,\]
moreover
\[\sigma = \frac{1}{2}\frac{1}{\|Hf(x_i)^{-1}\|} \geq \frac{1}{2}\gamma(K, \varepsilon).\]
We have $$\|D(Df)(x) - D(Df)(x_i)\| \leq (K + \varepsilon)\|x - x_i\|,$$ hence with $\|x-x_i\| \leq \frac{\sigma}{K+\varepsilon}$, and $\sigma = \frac{1}{2}\gamma(K, \varepsilon)$, we obtain $$D(Df)(x) \in D(Df)(x_i) + \sigma \mathcal{B}_{n\times n}.$$
Therefore, if $r = \min(\frac{\sigma}{K+\varepsilon}, \frac{1}{\sigma n}\gamma(K, \varepsilon))$, then
\[D(Df)(x) \in D(Df)(x_i) + \sigma \mathcal{B}_{n\times n}, \ \textrm{for all} \ x \in \mathbf{B}_r^n(x_i).\]
Hence, applying Theorem \ref{dl2}, there exist neighborhoods $U$ and $V$ of $x_i$ and $Df(x_i)$, respectively, such that $Df$ is invertible, with
\[U = \mathbf{B}_{\frac{r\sigma}{2(K+ \varepsilon)}}^n\left( x_i\right) , \ \ V=\mathbf{B}_{\frac{r\sigma}{2}}^n\left( 0\right).\]
So, with $\ds\eta(K, \varepsilon) = \frac{r\sigma}{2}=\frac{1}{4}r\gamma(K, \varepsilon)$, as $\|\textrm{grad} f(x)\| \leq \eta(K, \varepsilon)$ we have
\[x \in \mathbf{B}_{\frac{r\sigma}{2(K+ \varepsilon)}}^n\left( x_i\right)  \subset \mathbf{B}_{\psi_3(K, \varepsilon)}^n\left( x_i\right).\]
\end{proof}

{\bf Received: September, 2011}


\begin{thebibliography}{HD}
\bibliographystyle{plain}
\bibitem{C1} F. H. Clarke, {On the inverse function theorem,}
{Pacific Journal of Mathematics}, Vol 64, No 1 (1976), 97-102.
\bibitem{C2} F. H. Clarke, {Generalized gradients and applications,}
{Trans. Amer. Math. Soc.}, Vol 205 (1975), 247-262.
    \bibitem{F}  H. Federer, {Geometric measures theory},
Springer-Verlag, 1969.
\bibitem{G-L} G. H. Golub and C. F. van Loan , {Matrix computation}, Johns Hopkins Univ. Press 1983.
        \bibitem{Hir} M. W. Hirsch, {Differential Topology}, Springer-Verlag, New York - Heidelberg - Berlin, 1976.
         \bibitem{Morse}  M. Morse, {The critical points of a function of $n$ variables,} {Trans. Amer. Math. Soc.}, 33, (1931), 71-91.
        \bibitem{Morse1}  A. Morse, {The behavior of a function on its critical set,} {Ann. Math}, 40 (1939), 62-70.
           \bibitem{N}  L. Niederman, {Prevalence of exponential stability among nearly intergrable Hamiltonian systems,} Ergodic Theory and Dynamical Systems, (2007), 25p.
           \bibitem{P} P. Phien, {Some quantitative results on Lipschitz inverse and implicit functions theorems}, East-West J. Math. Vol. 13, No 1 (2011), 7-22.
  \bibitem{R} A. Rohde, {On the $\varepsilon$-Entropy of Nearly Critical Values}, {Journal of Approximation Theory}, 76 (1994), 166-194.
    \bibitem{Sa}  A. Sard, {The measure of the critical values of differentiable maps},
  { Bull. Amer. Math. Soc.} 48, (1942), 883-890.
    \bibitem{Y1} Y. Yomdin, {Some quantitative results in singularity theory}, {Anales Polonici Mathematici}, 37 (2005),  277-299.
    \bibitem{Y2} Y. Yomdin, {The Geometry of Critical and Near-Critical Values of Differentiable Mappings}, {Math. Ann.} 264, (1983), 495-515.
    \bibitem{Y3} Y. Yomdin, {Metric properties of semialgebraic
    sets and mappings and their applications in smooth analysis,}
    (Proceedings of the Second International Conference on Algebraic
    Geometry, La Rabida, Spain, 1984, J.M. Aroca, T. Sahcez-Geralda,
    J.L. Vicente, eds.), {Travaux en Cours, Hermann, Paris} (1987),
    165-183.
    \bibitem{Y-C} Y. Yomdin and G. Comte, {Tame geometry with application in smooth analysis},
     LNM vol. 1834, 2004.
\end{thebibliography}
\end{document}